\documentclass[11pt,a4paper]{article}
\usepackage[top=1in, bottom=1.25in, left=1in, right=1in]{geometry}
\usepackage[utf8]{inputenc}
\usepackage[english]{babel}
\usepackage{graphicx,epstopdf}
\usepackage{caption, subcaption,float}
\usepackage{color}
\usepackage{multirow}
\usepackage{tikz}
\usepackage{amsthm}
\usepackage{tabls}
\usepackage{cite}
\usepackage[round,sort,comma]{natbib}

\usepackage[ruled, linesnumberedhidden,vlined]{algorithm2e} 
\usetikzlibrary{matrix}
\usepackage{amsmath,amssymb,dsfont,mathtools}
\usepackage{hyperref}
\hypersetup{
    colorlinks=true,       
    linkcolor=blue,          
    citecolor=olive,       
    filecolor=magenta,      
    urlcolor=blue          
}
\usepackage{aliascnt}

\begin{document}

\newcommand{\ca}{\mathcal{C}_{AL}}

\title{Parabolic subgroups acting on the additional length graph}
\date{\today }
\author{Yago Antol\'in and Mar\'{i}a Cumplido}
%%%%

%%%%%%

\maketitle
% ------- Theorem styles -------
\theoremstyle{plain}
\newtheorem{theorem}{Theorem}

\newaliascnt{lemma}{theorem}
\newtheorem{lemma}[lemma]{Lemma}
\aliascntresetthe{lemma}
\providecommand*{\lemmaautorefname}{Lemma}

\newaliascnt{proposition}{theorem}
\newtheorem{proposition}[proposition]{Proposition}
\aliascntresetthe{proposition}
\providecommand*{\propositionautorefname}{Proposition}

\newaliascnt{corollary}{theorem}
\newtheorem{corollary}[corollary]{Corollary}
\aliascntresetthe{corollary}
\providecommand*{\corollaryautorefname}{Corollary}

\newaliascnt{conjecture}{theorem}
\newtheorem{conjecture}[conjecture]{Conjecture}
\aliascntresetthe{conjecture}
\providecommand*{\conjectureautorefname}{Conjecture}

\theoremstyle{remark}

\newaliascnt{claim}{theorem}
\newaliascnt{remark}{theorem}
\newtheorem{claim}[claim]{Claim}
\newtheorem{remark}[remark]{Remark}
\newaliascnt{notation}{theorem}
\newtheorem{notation}[notation]{Notation}
\aliascntresetthe{notation}
\providecommand*{\notationautorefname}{Notation}

\aliascntresetthe{claim}
\providecommand*{\claimautorefname}{Claim}

\aliascntresetthe{remark}
\providecommand*{\remarkautorefname}{Remark}

\newtheorem*{claim*}{Claim}
\theoremstyle{definition}

\newaliascnt{definition}{theorem}
\newtheorem{definition}[definition]{Definition}
\aliascntresetthe{definition}
\providecommand*{\definitionautorefname}{Definition}

\newaliascnt{example}{theorem}
\newtheorem{example}[example]{Example}
\aliascntresetthe{example}
\providecommand*{\exampleautorefname}{Example}

\def\autorefspace{\hspace*{-0.5pt}}
\def\sectionautorefname{Section\autorefspace}
\def\subsectionautorefname{Section\autorefspace}
\def\subsubsectionautorefname{Section\autorefspace}
\def\figureautorefname{Figure\autorefspace}
\def\subfigureautorefname{Figure\autorefspace}
\def\tableautorefname{Table\autorefspace}
\def\equationautorefname{Equation\autorefspace}
\def\Itemautorefname{item\autorefspace}
\def\Hfootnoteautorefname{footnote\autorefspace}
\def\AMSautorefname{Equation\autorefspace}

\newcommand{\co}{\simeq_c}
\newcommand{\w}{\widetilde}
\newcommand{\po}{\preccurlyeq}
\newcommand{\so}{\succcurlyeq}
\newcommand{\dist}{\mathrm{d}}

\def\Z{\mathbb Z} 
\def\Ker{{\rm Ker}} \def\R{\mathbb R} \def\GL{{\rm GL}}
\def\HH{\mathcal H} \def\C{\mathbb C} \def\P{\mathbb P}
\def\SSS{\mathfrak S} \def\BB{\mathcal B} 
\def\supp{{\rm supp}} \def\Id{{\rm Id}} \def\Im{{\rm Im}}
\def\MM{\mathcal M} \def\S{\mathbb S}
\newcommand{\bigveer}{\bigvee^\Lsh}
\newcommand{\wedger}{\wedge^\Lsh}
\newcommand{\veer}{\vee^\Lsh}
\def\diam{{\rm diam}}

\begin{abstract}
Let $A\neq A_1, A_2, I_{2m}$ be an irreducible Artin--Tits group of spherical type. We show that the  periodic elements of~$A$ and the elements preserving some parabolic subgroup of~$A$ act elliptically on the additional length graph~$\mathcal{C}_{AL}(A)$, a hyperbolic, infinite diameter  graph associated to~$A$ constructed by Calvez and Wiest to show that~$A/Z(A)$ is acylindrically hyperbolic.
We use these results to find an element $g\in A$ such that $\langle P,g \rangle\cong P* \langle g \rangle$ for every proper standard parabolic subgroup~$P$ of~$A$.
The length of $g$ is uniformly bounded with respect to the Garside generators, independently of~$A$. This allows us to show that, in contrast with the Artin generators case, the sequence  $\{\omega(A_n,\mathcal{S})\}_{n\in \mathbb{N}}$ of exponential growth rates of braid groups, with respect to the Garside generating set, goes to infinity.

\medskip

{\footnotesize

}

\end{abstract}

\section{Introduction}

It is well known that the braid group with $n+1$~strands, $A_n$, acts by isometries on the curve complex of the $n$-punctured disk, $\mathcal{D}_n$. This fact comes from the topological definition of~$A_n$, which says that $A_n$ is the mapping class group of~$\mathcal{D}_n$. We know a lot about this curve complex, including its $\delta$-hyperbolicity, which makes it a fundamental tool when proving properties of the braid group.

\medskip
On the other hand, braid groups also belong to a family of presentable groups, called Artin--Tits groups \citep{Artin1947}. To define them we need a finite set of generators~$\Sigma$ and a symmetric matrix  $M=(m_{s,t})_{s,t\in \Sigma}$ with $m_{s,s}=1$ and $m_{s,t}\in\{2,\dots, \infty \}$ for $s\neq t$. The Artin--Tits system associated to $M$ is $(A,\Sigma)$, where~$A$ is the so called Artin--Tits group presented in the following way:

$$A=\langle \Sigma \,|\, \underbrace{sts\dots}_{m_{s,t} \text{ elements}}=\underbrace{tst\dots}_{m_{s,t} \text{ elements}} \forall s,t\in \Sigma,\, s\neq t,\, m_{s,t}\neq \infty \rangle.$$
\noindent Notice that relations in this presentation contain only positive powers of the generators. This allows us to define $A^+$ as the positive monoid given by the same presentation. We also can obtain the Coxeter group~$W_A$ associated to $(A,\Sigma)$ by adding the relations $s^2=1$:
$$W_A=\langle \Sigma \,|\, s^2=1 \, \forall s\in \Sigma ; \underbrace{sts\dots}_{m_{s,t} \text{ elements}}=\underbrace{tst\dots}_{m_{s,t} \text{ elements}} \forall s,t\in \Sigma,\, s\neq t,\, m_{s,t}\neq \infty \rangle.$$
If~$W_A$ is finite, the corresponding Artin--Tits group (or Artin--Tits system) is said to be of spherical type. If~$A$ cannot be decomposed as a direct product of non-trivial Artin--Tits groups, we say that~$A$ is \emph{irreducible}. Irreducible Artin--Tits groups of spherical type are completely classified (see  \autoref{coxeter}) in ten classes \citep{Coxeter}. The main example on these groups is the braid group~$A_n$, which is provided by the presentation

\begin{figure}[h]
  \centering
  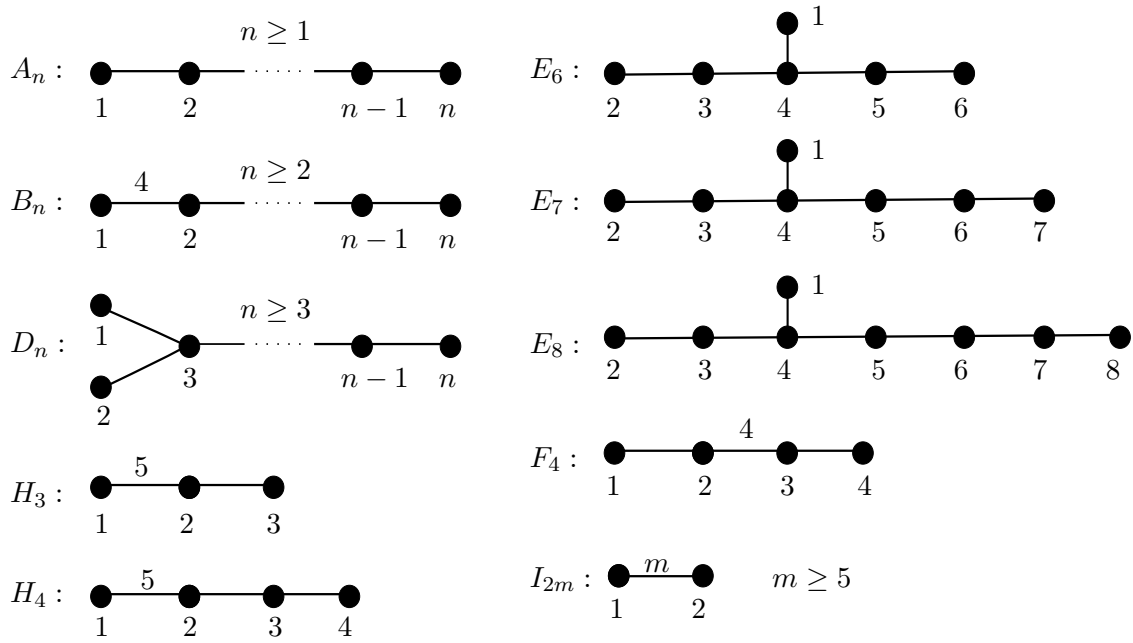
  \medskip
  \caption{Classification of irreducible Coxeter graphs of finite type}
  \label{coxeter}
\end{figure}

\begin{equation*}
A_n=\left\langle \sigma_1,\dots, \sigma_{n}\, \begin{array}{|lr}
                                              \sigma_i\sigma_j=\sigma_j\sigma_i, & |i-j|>1 \\
                                                \sigma_i\sigma_{j}\sigma_i= \sigma_{j}\sigma_{i}\sigma_{j} , & |i-j|=1
                                             \end{array}
 \right\rangle .\end{equation*}

\medskip\noindent
Artin--Tits groups of spherical type share many properties with~$A_n$, but to prove them we cannot use arguments that involve the curve complex. To overcome this difficulty, new complexes related with all Artin--Tits groups of spherical type have been introduced.

\medskip
Let $(A,\Sigma)$ be an Artin--Tits system of spherical type.
On the one hand, in \citep{calvez2016graph} the authors constructed the additional length graph of~$A$, noted~$\ca(A)$. As we will see later, this complex relies on technical concepts about Garside theory and it is in fact defined for every Garside group. 
The interest on~$\ca(A)$ lies on the fact that it is $\delta$-hyperbolic when $A$ is irreducible. Calvez and Wiest conjectured that~$\ca(A_n)$ is quasi-isometric to the curve graph (the 1-skeleton  of the curve complex) of~$\mathcal{D}_n$ and they even made a step forward proving that the braids that act loxodromically on~$\ca(A_n)$ are pseudo-Anosov. We recall that an element $\alpha$ acts \emph{loxodromically} on~$\ca(A_n)$ if the orbit of every element of $\ca(A)$ by $\alpha$ is quasi-isometric to $\mathbb{Z}$. On the contrary, $\alpha$ acts \emph{ellipticaly} on $\ca(A_n)$ if its orbits are bounded. 

\begin{proposition}[{\citealp[Proposition 2]{calvez2016graph}}]\label{prop: calvez elliptic}
We consider the action of the braid group $A_n$ on its additional length graph~$\ca(A_n)$ by left multiplication. Then periodic and reducible elements act elliptically.
\end{proposition}

On the other hand, we have the complex of irreducible parabolic subgroups~$\mathcal{P}(A)$, defined in \citep{CGGW}.  A \emph{standard parabolic subgroup}, $A_X$, is the  subgroup generated by some $X\subseteq \Sigma$. A subgroup~$P$ is called \emph{parabolic} if it is conjugate to a standard parabolic subgroup, that is, $P=\alpha^{-1} A_Y \alpha$ for some standard parabolic subgroup~$A_Y$ and some $\alpha\in A$. We denote by $z_P$ the unique positive element that generates the center $Z(P)$ of $P$. The vertices of~$\mathcal{P}(A)$ are the irreducible ones, that is, the parabolic subgroups that cannot be decomposed as a direct product of non-trivial parabolic subgroups. A set of vertices $\{P_0,\cdots, P_n\}$ spans an $n$-simplex if $z_{P_i}z_{P_j}=z_{P_j}z_{P_i}$ for all $i\neq j$. In \citep[Theorem~2.2]{CGGW}, it is proven that having $z_{P_i}z_{P_j}=z_{P_j}z_{P_i}$ is equivalent to have one of these three situations:

\begin{itemize}
\item $P_i \subset P_j$;
\item $P_j \subset P_i$;
\item $P_i \cup P_j = \{ 1\}$ and $p_ip_j=p_jp_i$, for every $p_i\in P_i$, $p_j\in P_j$.
\end{itemize}

\medskip
In the braid case, the subset of irreducible proper parabolic subgroups is in bijection with the isotopy class of curves in~$\mathcal{D}_n$. This makes~$\mathcal{P}(A_n)$ isomorphic to the curve complex of~$\mathcal{D}_n$. This complex seems then more natural for studying Artin-Tits groups than~$\ca(A)$. However, the hyperbolicity of the complex of irreducible parabolic subgroups is not proven yet. One approach to solve that problem could be to find links between the 1-skeleton of the complex of irreducible parabolic subgroups (which is in fact a flag complex) and the additional length graph, which is precisely what we will do in this article. We will generalize \autoref{prop: calvez elliptic} of Calvez and Wiest to irreducible Artin--Tits groups of spherical type. Due to the bijection between the curve complex and the complex of irreducible parabolic subgroups mentioned before, we can realise that the reducible braids correspond to elements that preserve a family of parabolic subgroups that form a simplex in $\mathcal{P}(A)$. Notice that the action of~$A$ on~$\mathcal{P}(A)$ is induced by the conjugation action of~$A$ on itself.
We say that $x\in A$ is \emph{periodic} if some of its power acts trivially on~$\mathcal{P}(A)$. 

\begin{theorem}\label{th:main}
Let $A\neq A_1, A_2, I_{2m}$ be an irreducible Artin--Tits group of spherical type. The periodic elements of~$A$ and the elements preserving some simplex of~$\mathcal{P}(A)$ (i.e. normalizing parabolic subgroup of~$A$) act elliptically on~$\ca(A)$.
\end{theorem}

As explained in \citep[Proposition~4.9]{CW2019}, the proof of this result (and more precisely \autoref{corol}) is the key to prove that there is a 9-Lipschitz function from~$\mathcal{P}(A)$ to~$\ca(A)$, when~$A$ has rank at least~3. As a consequence, we can prove that in these cases~$\mathcal{P}(A)$ has infinite diameter \citep[ Corollary~4.10]{CW2019}.

\bigskip
For the second part of this article, we will use one of the key features of the graph~$\ca(A)$: its hyperbolicity constant is independent of the Artin-Tits group of spherical type. Moreover, when proving \autoref{th:main} we will find bounds on the diameter of orbits of parabolic subgroups acting on~$\ca(A)$ that are again independent of the group~$A$.
Combining this uniformity of constants and standard techniques of groups acting on hyperbolic spaces
we will show that we can find a common ``free-product complement'' for all standard proper parabolic subgroups. Namely,

\begin{theorem}\label{thm: free product}
There exists  a constant~$K$ such that for every irreducible Artin-Tits group $(A,\Sigma)$ of spherical type  $A\neq A_1, A_2, I_{2m}$ there is an element $g_*\in A^+$ such that \begin{enumerate}
\item[(1)] the element~$g_*$ has length at most~$K$ with respect to the Garside generators (see \autoref{garsidetheory}), and
\item[(2)]  for every proper standard parabolic subgroup~$A_X$ of~$A$%( $X\subsetneq \Sigma$) 
,  one has that $\langle g_*, A_X\rangle \cong \langle g_* \rangle * A_X$.
\end{enumerate}
\end{theorem}

\medskip

The Cayley graph of an Artin--Tits group of spherical type is usually better understood with respect to the Garside generators (explained in \autoref{garsidetheory}) than with respect to the Artin generating set~$\Sigma$.
Let~$\mathcal{S}$ denote the set of Garside generators.
Then, for every parabolic subgroup~$A_X$, one has that the natural subgroup inclusion induces a graph isometric inclusion $\Gamma(A_X,A_X\cap \mathcal{S}^{\pm 1})\to \Gamma(A, \mathcal{S}^{\pm 1})$. Here,~$\Gamma(G,X)$ denotes the Cayley graph of a group~$G$ with respect to a generating set~$X$.
We will use this isometric inclusion together with the ``free-product complement'' to derive that the (relative) growth rate of proper parabolic subgroups is strictly smaller than the ambient group.

\medskip
\noindent
Before stating this last result, let us fix some notation. Let~$M$ be a monoid, $X$~a finite generating set of~$M$. We denote the (relative) exponential growth rate of~$M$ by
$$\omega(M,X)= \lim_{n\to \infty}\left( \sharp\{g \in M : |g|_X \leq n\}\right)^{\frac{1}{n}}, $$
where~$|g|_X$ denotes the length of the shortest word in~$X$ representing~$g$.
Notice that this limit exists thanks to the sub-multiplicativity of the word length and Fekete's lemma.

\begin{corollary}\label{cor: growth}
Let $A\neq A_1$ be an irreducible Artin--Tits group of spherical type. Let~$\mathcal{S}_A$ be the Garside generating set of $A$. 
For every proper parabolic subgroup~$A_X$ of~$A$, one has that 
$$
\omega(A_X,\mathcal{S}_A^{\pm 1})<\omega(A,\mathcal{S}_A^{\pm 1})
\quad \text{ and } \quad
\omega(A_X^+, \mathcal{S}_A)<\omega(A^+,\mathcal{S}_A).
$$
Moreover, the sequences $\{\omega(A_n,\mathcal{S}_{A_n}^{\pm 1})\}_{n=1}^{\infty}$, $\{\omega(B_n,\mathcal{S}_{B_n}^{\pm 1})\}_{n=2}^{\infty}$, $\{\omega(D_n,\mathcal{S}_{D_n}^{\pm 1})\}_{n=3}^{\infty}$ (and the corresponding sequences for the  submonoid of positive elements) are increasing and unbounded. 
\end{corollary}

\noindent
This result contrasts with the case of standard Artin generators. 
In that case, it is known that both $\{\omega(A_n^+,\Sigma)\}_{n=1}^{\infty}$ and $\{\omega(A_n,\Sigma)\}_{n=1}^{\infty}$ are increasing and converge. More specifically, for the submonoid of positive elements, a beautiful recent result of \cite[Theorem 6.8]{FG-M} shows that $\{\omega(A_n^+,\Sigma)\}_{n=1}^{\infty}$ converges to the KLV-constant $q_\infty= 3.23363...$ .

\section{Preliminaries}

\subsection{Garside theory}\label{garsidetheory}

Let us briefly recall some concepts from Garside theory (for a general reference, see \cite{Dehornoy1999}). A group~$G$ is called a \emph{Garside group} with Garside structure $(G,\mathcal{M},\Delta)$ if it admits a submonoid~$\mathcal{M}$ of positive elements such that $\mathcal{M}\cap \mathcal{M}^{-1}=\{1\}$ and a special element $\Delta \in \mathcal{M}$, called Garside element, with the following properties:

\begin{itemize}
 \item There is a partial order in~$G$, $\po$,  defined by $a \po b \Leftrightarrow a^{-1}b \in \mathcal{M}$ such that for all $a,b\in G$ there exists a unique gcd, denoted $a \wedge b$, and a unique lcm, denoted $a \vee b$, with respect to~$\po$.  This order is called prefix order and it is invariant under left-multiplication.

\item The set of simple elements $\mathcal{S}\coloneqq \{s\in G\,|\, 1\po s \po \Delta  \}$ generates~$G$. These are also called \emph{Garside generators}.

\item $\Delta^{-1}\mathcal{M} \Delta = \mathcal{M}$. 

\item $\mathcal{M}$ is atomic: If we define the set of atoms  as the set of elements $a\in \mathcal{M}$ such that there are no non-trivial elements $b,c\in\mathcal{M}$ such that $a=bc$, then for every $x\in\mathcal{M}$ there is an upper bound on the number of atoms in a decomposition of the form $x=a_1a_2\cdots a_n$, where each~$a_i$ is an atom.  

\end{itemize}

\noindent
In a Garside group, the monoid~$\mathcal{M}$ also induces a partial order which is invariant under right-multiplication, the suffix order $\so$. This order is defined by ${a\so b} \Leftrightarrow ab^{-1}\in \mathcal{M}$, and for all $a,b\in G$ there exists a unique gcd $(a\wedger b)$ and a unique lcm $(a\veer b)$ with respect to~$\so$. 

\medskip
We say that a Garside group has \emph{finite type} if~$\mathcal{S}$ is finite. It is well known that every Artin--Tits group of spherical type $A$ admits a Garside structure of finite type where the monoid $\mathcal M$ is precisely $A^+$ \citep{Brieskorn1972,Dehornoy1999}. The monoid~$A^{+}$ injects on~$A$ \citep{Par}, which implies that the atoms of~$A$ are precisely the generators in the presentation given in the introduction.

\begin{remark}\label{permutation}
The conjugate by~$\Delta$ of an element~$x$ will be denoted $\tau(x) =\Delta^{-1}x\Delta.$ Notice that $\Delta^{-1}\mathcal{M} \Delta = \mathcal{M}$ implies that the set of prefixes of $\Delta$ equals its set of suffixes and then~$\Delta$ is decomposed as $a\cdot b$, where~$a$ is an atom, if and only if we can write $\Delta= b \cdot a'$, where~$a'$ is an atom. This means that~$\tau$ provides a permutation of the atoms of the Garside group.
\end{remark}

\begin{proposition}[{\citealp[Lemma~5.1, Theorem~7.1]{Brieskorn1972}}]\label{definicion_delta}
Let $(\Sigma,A)$ be an Artin--Tits system of spherical type. Then the Garside element for $A$ is $$\Delta= \bigvee_{\sigma_i\in \Sigma}(\sigma_i)={\bigvee_{\sigma_i\in \Sigma}}^{\Lsh}(\sigma_i).$$
Moreover, the conjugation by $\Delta^2$ is trivial, that is $\tau^2=Id$.
\end{proposition}

\begin{lemma}\label{delta_reverse}Let $x$ be an element of an Artin--Tits system $(A,\Sigma)$ of spherical type. Let $x= a_1 a_2\cdots a_r$ with $a_i\in \Sigma \cup \Sigma^{-1}$ and define $\overleftarrow{x}:= a_r a_{r-1}\cdots a_{1}$. Then, $\overleftarrow{\cdot }\colon A \to A$ is an involution  and a well defined anti-homomorphism (for all $x,y\in A$ $\overleftarrow{xy}=\overleftarrow{y}\overleftarrow{x}$). In particular, $\Delta = \overleftarrow{\Delta}$.
\end{lemma}

\begin{proof}
We have a well defined anti-homomorphism because the relations in the presentation of an Artin--Tits group of spherical type are symmetric, and so, for every $x,y\in A$, we have that $\overleftarrow{xy}=\overleftarrow{y}\overleftarrow{x}$. Notice that, thanks to this symmetry, $(a\vee b)= \overleftarrow{\left(\overleftarrow{a} \veer \overleftarrow{b}\right)}$, for every $a,b\in A^+$.
Then, since our anti-homomorphism preserves atoms, $$\bigvee_{\sigma_i\in \Sigma}(\sigma_i) = \overleftarrow{\left({\bigvee_{\sigma_i\in \Sigma}}^\Lsh(\overleftarrow{\sigma_i})\right)} = \overleftarrow{\left({\bigvee_{\sigma_i\in \Sigma}}^\Lsh(\sigma_i)\right)} .$$
This fact together with \autoref{definicion_delta} implies that 
 $\Delta = \overleftarrow{\Delta}$.
\end{proof}

%\begin{definition}We define the right complement of a simple element $a$ as $\partial(a)=a^{-1}\Delta$ and the left complement as $\partial^{-1}(a)=\Delta a^{-1}$.\end{definition}

%\begin{remark} Observe that $\partial^2 =\tau$ and that, if $a$ is simple, then $\partial (a)$ is also simple, i.e., $1~\po~\partial(a)~\po~\Delta$. Both claims follow from $\partial(a)\tau(a)=\partial(a)\Delta^{-1}a\Delta = \Delta$ since $\partial(a)$ and $\tau(a)$ are positive.\end{remark}

\begin{definition}
We say that the product of two simple elements $a\cdot b$,  is \emph{left-weighted} (resp. \emph{right-weighted}) if $ab\wedge \Delta =a$ (resp. $ab\wedger \Delta =b$). 
\end{definition}

\begin{remark}\label{libres}
In an Artin--Tits group of spherical type, the simple elements of the Garside structure are the square-free ones, that is, every positive word representing that element does not contain the square of an atom \citep{Brieskorn1972,Del}. This implies that  $a\cdot b$ is left-weighted (resp. \emph{right-weighted}) if for every atom $t$ such that $t\po b$ (resp. $a\so t$), we have that $a\so t$ (resp. $t\po b$).
\end{remark}

\begin{definition}
We say that $x = \Delta^k s_1\cdots s_r$ is in \emph{left normal form} if $k\in \mathbb{Z}$, $s_i\notin \{1,\Delta\}$ is a simple element for $i=1,\ldots , r$, and $s_i \cdot s_{i+1}$ is left-weighted for $0<i < r$. Analogously, $x =  s_1\cdots s_r\Delta^k$ is in \emph{right normal form} if $k\in \mathbb{Z}$, $s_i\notin \{1,\Delta\}$ is a simple element for $i=1,\ldots , r$, and $s_i s_{i+1}$ is right-weighted for $0<i < r$. When the right and the left normal form coincide, we will just refer to the \emph{normal form}. 
\end{definition}

It is well known that the normal forms of an element are unique \cite[Corollary~7.5]{Dehornoy1999} and that the numbers~$r$ and~$k$ do not depend on the normal form (left or right). We define the \emph{infimum}, the \emph{canonical length} and the \emph{supremum} of~$x$ respectively as  $\inf(x)=k$, $\ell(x)=r$ and $\sup(x)=k+r$. Equivalents definitions of supremum and infimum are 

\[ \inf(x)=\max\{p\,|\, \Delta^p\po x\} \quad \text{and} \quad \sup(x)=\min\{p\,|\, x\po \Delta^p\}. \]

%\medskip
%\noindent
%For the forthcoming proofs, it will be useful to know how to construct normal forms. Notice that $a\cdot b$ is left-weighted if $s=a^{-1} (ab \wedge \Delta)$ is trivial. If that prefix $s \po b$ is not trivial, we shall ``slide''~$s$ to the end of $a$ in order to have a left-weighted product. That is, if we write $b=s\cdot t$ the factorization $(as)\cdot t$ is left-weighted thanks to the maximality of~$s$. The process of ``sliding'' the factor~$s$ is called \emph{local sliding} of $a\cdot b$. The following lemma tell us that, if we have an element that is already in left normal form and add a simple factor, we only need to perform a ``round'' of local slidings:
%
%\begin{lemma}[{\citealp[Proposition~3.1]{Charney1992}}]\label{normal}
%Let $s_1s_2\cdots s_k$ be a left normal form and let $s'_{k+1}$ be a simple element. Consider the product $s_1\cdots s_k s'_{k+1}$ and for $i=k,\dots,1$ apply a local sliding to $s_{i}s'_{i+1}$, that is, if $t_i= s_i^{-1}(s_is'_{i+1}\wedge \Delta)$, define  $s'_i=s_it_i$, $s''_{i+1}=t_i^{-1}s'_{i+1}$ and $s''_1=s'_1$. Then we have that $s''_1\cdots s''_{k+1}$ is the left normal form of $s_1\cdots s_k s'_{k+1}$.
%\end{lemma}

\subsection{The additional length graph}

The construction of the additional length graph is made for any Garside group and its key ingredient is the use of absorbable elements, which are defined below. 

\begin{definition}[{\citealp[Definition~1]{calvez2016graph}}]
Let~$G$ be a Garside group. We say that $y\in G$ is an \emph{absorbable element} if the two following conditions are satisfied:

\begin{enumerate}
\item $\inf(y)=0$ or $\sup(y)=0$.

\item There is a $x\in G$ such that $\inf(xy)=\inf(x)$ and $\sup(xy)=\sup(x)$.
\end{enumerate}
\noindent
In this case we say that \emph{$x$ absorbs $y$}.
\end{definition}

\noindent
We say that $x$ absorbs $y$ because the length of the normal form of $xy$ is the same length of the normal form of $x$. So, loosely speaking the normal form of $x$ ``absorbs'' the factors of the normal form of $y$.

\medskip

\noindent
\emph{Example.} 
Consider the braid group $A_n$, $n>2$.
As $\sigma_1$ commutes with $\sigma_3$, if we let $x= \sigma_1 \cdots \sigma_1$ and $y= \sigma_3 \cdots  \sigma_3$, we have that  $xy=(\sigma_1\sigma_3) \cdots (\sigma_1 \sigma_3)$. Hence, $x$ absorbs $y$.

On the contrary, if $x=\Delta \sigma_i^{-1}$ and $y=\sigma_i$, then $x$ does not absorb $y$ because $\inf(x)=0$ and $\inf(xy)=1$.

\begin{definition}[{\citealp[Definition~2]{calvez2016graph}}]
Let $(G, G_+, \Delta)$ be a Garside structure.  We define the \emph{additional length graph} of~$G$, $\ca(G)$ as follows:

\begin{itemize}

\item The vertices are in one-to-one correspondence with $G/ \langle \Delta \rangle$, that is, the equivalence classes $g\Delta^\Z=\{g\Delta^p\,|\, p\in\Z\}$. Every class~$v$ has a unique representative with infimum~0, denoted by $\overline{v}$.

\item Two vertices $v=\overline{v}\Delta^\Z$ and $w=\overline{w}\Delta^\Z$ are connected by an edge if and only if we have one of the two following situations:
\begin{enumerate}

\item There is a simple element $m\neq 1,\Delta$ such that $\overline{v}m\in w$.

\item There is an absorbable element $y\in G$ such that $\overline{v}y\in w$.  
\end{enumerate}
\end{itemize}
\end{definition}

\noindent
We give a metric structure to this complex by saying that the length of every edge in the graph is 1. We denote the distance between two vertices $v,w$ by $\dist_{\ca} (v,w)$.

\section[Normalizers of parabolic subgroups act eliptically]{ Normalizers of parabolic subgroups act eliptically}

The aim of this section is to prove \autoref{th:main}, that is, we prove that the normalizers of parabolic subgroups of an irreducible spherical Artin-Tits group $A\neq A_1, A_2, I_{2m}$ act elliptically on the additional length graph.

\medskip
In \citep{Vanderlek1983} it is proven that any standard parabolic subgroup~$A_X$ of an Artin--Tits group of spherical type is an Artin--Tits group of spherical type itself. This means that~$A_X$ has also a Garside structure, whose Garside element is denoted by~$\Delta_X$ and equals the least common multiple of the elements in~$X$. We denote by $\tau_X$ the conjugation by $\Delta_X$.

\subsection{Ribbons}

We will use some objects defined in \citep{Cumplido2017b} and \citep{CGGW} that we call ribbons. We shall remark that these ribbons are slightly different from the classical concept of ribbon introduced in \citep{Godelle2003}. 

\begin{definition}
Let $(A,\Sigma)$ be an Artin--Tits system of spherical type and let $X\subsetneq \Sigma$, $t\in \Sigma$. We define 
$$r_{X,t}=\Delta_{X\cup \{t\}}\Delta^{-1}_{X},\qquad r_{t,X}=\tau_{X\cup\{t\}}(r_{X,t})=\Delta^{-1}_{X}\Delta_{X\cup \{t\}}.$$
We will say that~$r_{X,t}$ is a \emph{right-ribbon} and~$r_{t,X}$ is a \emph{left-ribbon}.

\end{definition}

\begin{remark}\label{remark_ribbon}
Notice that  $\Delta_{X\cup \{t\}}=r_{X,t}\Delta_X$ is simple and then it is square-free (\autoref{libres}).
As~$\Delta_X$ can start with any letter of~$X$, if $t\notin X$, the only suffix letter of~$r_{X,t}$ is~$t$. Analogously, $t$ is the only prefix letter of~$r_{t,X}$.
\end{remark}

\begin{lemma}\label{ribbons_normal_form}
Let $(A,\Sigma)$ be an Artin--Tits system of spherical type and let $X\subset \Sigma$. Then $r_{t,X}=\overleftarrow{\,r_{X,t}\,}$. In particular, if $t\notin X$, then both $r_{X,t} \cdot r_{t,X}$ and $r_{t,X} \cdot r_{X,t}$ are left and right-weighted.
\end{lemma}

\begin{proof} 
By \autoref{delta_reverse}, we have that $$\overleftarrow{ \, r_{X,t}\,}= \overleftarrow{\Delta^{-1}_{X}}\cdot \overleftarrow{\Delta_{X\cup \{t\}}}= \Delta^{-1}_{X}\Delta_{X\cup \{t\}}=r_{t,X}.$$
Also notice that, as we have seen in that lemma, the atoms that are suffixes of~$x$ coincide with the atoms that are prefixes of~$\overleftarrow{x}$ and viceversa. Then, by \autoref{libres}, $r_{X,t} \cdot r_{t,X}$ and $r_{t,X} \cdot r_{X,t}$ are both left and right-weighted.
\end{proof}

\begin{remark}\label{permutation2}
Notice that by definition, conjugations by~$r_{X,t}$ and~$r_{t,X}$  are equivalent to applying $ \tau_X\circ\tau_{X\cup \{t\}}$ and $ \tau_{X\cup \{t\}} \circ \tau_X$ respectively (recall that $\tau^2=Id$). So by \autoref{permutation}, conjugation by~$r_{X,t}$ and~$r_{t,X}$ induces a permutation of the atoms of $X\cup\{t\}$. This implies that there exists a unique $Y\subset X\cup\{t\}$ such that $r_{X,t}X=Yr_{X,t}$ and $Xr_{t,X}=r_{t,X}Y$.  We say that~$r_{X,t}$ is an \emph{elementary $X$-ribbon-$Y$}.

Notice also that, since~$t$ is the only atomic prefix of $r_{t,X}$ (\autoref{remark_ribbon}), $r_{X,t}s$ is simple for every $s\in X$. So~$Y$ is formed by all atoms $u\in X\cup \{t\}$ such that $u\not\po r_{X,t}$ and $r_{t,X} \not\so u$ (\autoref{libres}). Moreover, $r_{X,t}$ has a unique atomic prefix and $r_{t,X}$ has a unique atomic suffix.
\end{remark}

\begin{definition} Let $(A,\Sigma)$ be an Artin--Tits system of spherical type and $X,Y\subsetneq \Sigma$. We say that $\alpha\in A$ is a \emph{$X$--ribbon--$Y$} if~$\alpha$ can be decomposed as a product of left-ribbons $r_1\cdots r_m$ and there exists a sequence of subsets of~$\Sigma$ of the form $X_1=X$, $X_2,\cdots,X_m, X_{m+1}=Y$ such that~$r_i$ is an elementary $X_i$-ribbon-$X_{i+1}$.
\end{definition}

%\noindent We say that an element $\alpha\in A$ can be \emph{constructed} with $k$~absorbable elements if~$\alpha$ can be written as a product of $k$~absorbable elements. 

\subsection{Proof of \autoref{th:main}}

We will prove that we can write any normalizer of a proper standard parabolic subgroup as the product of at most 9 absorbable elements.

\begin{lemma}\label{deltak}
 Let $(A,\Sigma)$, $A\neq  A_1, A_2, I_{2m}$, be an Artin--Tits system of spherical type with Garside element~$\Delta$  and let $X\subsetneq \Sigma$. Then, for every $k\in \Z$, $\Delta^k$ is a product of at most~3 absorbable elements and $\Delta_X^k$ is a  product of at most~2 absorbable elements.
\end{lemma}

\begin{proof}
Suppose that $k>0$. Take $A=\sigma_i^k$, $B=\sigma_j^k$ where $\sigma_i\in \Sigma$ and $\sigma_j\in \Sigma$ commute, and $C=B^{-1}\cdot A^{-1}\cdot \Delta^k$. We claim that $A$, $B$ and $C$ are absorbable. If this is true, $ABC$ is the desired decomposition for~$\Delta^k$. Also, by \cite[Lemma~1]{calvez2016graph},  $C^{-1}\cdot B^{-1} \cdot A^{-1}$ is the desired decomposition of $\Delta^k$ when $k<0$.

\medskip\noindent
Firstly, we have that $\inf(A)=\inf(B)=0$. We want to see that also $\inf(C)=0$. As $A\cdot B= (\sigma_i\sigma_j)^k$, we can write 
\begin{equation}\label{eq} C= \Delta \tau\left((\sigma_i\sigma_j)^{-1}\right) \cdot \Delta \tau^2\left((\sigma_i\sigma_j)^{-1}\right)\cdots\Delta \tau^k\left((\sigma_i\sigma_j)^{-1}\right).
\end{equation}
Notice that $\Delta \tau^{p+1}\left((\sigma_i\sigma_j)^{-1}\right)= \tau^{p}\left((\sigma_i\sigma_j)^{-1}\right) \Delta = \tau^{p}\left((\sigma_j\sigma_i)^{-1}\right) \Delta$. 
By \autoref{delta_reverse}, we have that for every $q>0$ and every atom $\sigma_m$, 
$$\Delta \tau^q\left((\sigma_i\sigma_j)^{-1}\right)\so\sigma_m \Longleftrightarrow \sigma_m\po  \tau^q\left((\sigma_j\sigma_i)^{-1}\right) \Delta,$$ 
meaning that $\Delta \tau^{p-1}\left((\sigma_i\sigma_j)^{-1}\right) \cdot \Delta \tau^p\left((\sigma_i\sigma_j)^{-1}\right)$ is left and right-weighted for every $p>0$. Hence, (\ref{eq}) is the normal form of~$C$ and $\inf(C)=0$.                                                                                                                                                                                                                                                                                                                                                                                                                                                                                                                                                                                                                                                                                                                                                                                                                                                                                                                                                                                                                                                                                                                                                                                                                                                                                                                                                                                                                                                                                                                                                                                                                                                                                                                                                                                                                                                                                                                                                                                                                                                                                                                                                                                                                                                                                                                                                                                                                                                                                                                                                                                                                                                                                                                                                                           

\medskip
\noindent
We can easily see that $A$ and $B$ absorb each other. Now let us see that $B$ absorbs $C$. We have that 
$$B \cdot C = (\sigma_i)^{-k}\cdot \Delta^k=  \Delta \tau(\sigma_i^{-1}) \cdot \Delta \tau^2(\sigma_i^{-1})\cdots\Delta \tau^k(\sigma_i^{-1}).$$
 As before, this expression is the normal form of $B\cdot C$. Then $\inf(B)=\inf(B\cdot C)= 0$  and $\sup(B)=\sup(B\cdot C)= k$, as desired.
This concludes that~$\Delta^k$ is a product of at most 3~absorbable elements.

\medskip
In most of the cases, $\Delta_X^k$ is absorbable itself. It suffices the existence of an atom $\sigma_j$ that commutes with~$X$ so that $\sigma_j^k$ absorbs $\Delta_X^k$. If it is not the case, we can take $\sigma_i\in X$ and $\sigma_j\not\in X$ such that $\sigma_i\sigma_j=\sigma_j\sigma_i$ and let $A=\sigma_i^k$ and $B=A^{-1} \Delta_X^k$. $A$ is absorbed by $\sigma_j^k$ and~$A$~absorbs~$B$. 
\end{proof}

\begin{proposition}\label{diameter}
Let $A\neq A_1,A_2, I_{2m}$ be any Artin--Tits group of spherical type.  Every element in a proper standard parabolic subgroup of~$A$ is a product of at most~3 absorbable elements. In particular, the orbit on~$\ca(A)$ of every proper standard parabolic subgroup of~$A$ has diameter at most~3.
\end{proposition}

\begin{proof}
Let~$A_X$ be a proper standard  parabolic subgroup of~$A$ and take $x\in A_X$. As~$A_X$ is an Artin--Tits group of spherical type, we can take the left normal form of~$x$ in~$A_X$, which is of the form $\Delta_X^k s_1\cdots s_l$.  By \autoref{deltak}, $\Delta_X^k$ is a product of at most~2 absorbable elements.
So, we assume that $x=s_1\cdots s_l$. We want to see that this element is absorbable.

\medskip
If there is $\sigma_i\in \Sigma \setminus X$ that commutes with~$X$, then~$\sigma_i^l$ absorbs $s_1\cdots s_l$. Otherwise, take an atom $t\in \Sigma\setminus X$ not commuting with $X$. We claim that \begin{equation}\label{eq2}
y=\tau_{X\cup \{t\}}^{l-1}(r_{X,t})\cdot \tau_{X\cup \{t\}}^{l-2} (r_{X,t})\cdots  \tau_{X\cup \{t\}}(r_{X,t})\cdot r_{X,t}
\end{equation} absorbs $s_1\cdots s_l$. Firstly, recall that $\tau_{X\cup \{t\}}(r_{X,t})=r_{t,X}$ and that by \autoref{ribbons_normal_form}, the expression~(\ref{eq2}) is the normal formal of~$y$, so $\inf(y)=0$ and $\sup(y)=l$.
As conjugations by~$r_{X,t}^{-1}$ and~$r_{t,X}^{-1}$  are equivalent to applying $ \tau_{X\cup \{t\}}\circ \tau_X$ and $\tau_X \circ\tau_{X\cup \{t\}} $ respectively, conjugation by $r_{X,t}^{-1}r_{t,X}^{-1}$ fixes every element in~$A_X$.
This allows us to write
\begin{equation}\label{eq3}
yx = s'_l\cdot s_{l-1}' \cdots s_1'
\end{equation}
where $s'_i= r_{X,t}\cdot s_{l-i+1}$ if~$i$ is odd and $s'_i= s_{l-i+1} \cdot r_{t,X}$ if~$i$ is even. We want to prove that equation (\ref{eq3}) expresses the normal form of~$yx$ and $s'_i\neq \Delta$ for $1\leq i\leq l$. To do that, it suffices to show the following:

\begin{enumerate}
\item  $sr_{t,X}$ and $r_{X,t}s$ are simple and different from~$\Delta$ for every $s\po \Delta_X$, $s\neq \Delta_X$.

\item If $a\cdot b$ is left-weighted for $a,b \po \Delta_X$, then $a r_{t,X} \cdot r_{X,t} b$ and  $r_{X,t}a \cdot b r_{t,X}$ are left-weighted.
\end{enumerate}

To prove the first statement, notice that there exists a simple element $s'\in \Delta_X$ such that $ss'=\Delta_X$. On the other hand, by definition $\Delta_{X\cup \{t\}}=\Delta_X r_{t,X}= s s' r_{t,X}$. By \autoref{permutation2}, there is a positive element $s''\in \Delta_{X\cup \{t\}}$ such that $s'r_{t,X}=r_{t,X}s''$, hence $\Delta_{X\cup \{t\}}= s r_{t,X}s''$. Hence $s r_{t,X}$ is simple, which is different from~$\Delta$ because $s\neq \Delta_X$. For~$r_{X,t}s$, the reasoning is analogous. using that $\Delta_X\so s$.

\medskip
Let us now prove the second statement. Let $u\in X\cup\{t\}$ be the only atomic suffix of $r_{t,X}$ (\autoref{permutation2}). By \autoref{ribbons_normal_form}, $ u \po r_{X,t}$, so take an atom $u'\neq u$, such that $u' \po r_{X,t} b$. By \autoref{permutation2}, we have $u'r_{X,t}= r_{X,t}\tau_{X}(\tau_{X\cup \{t\}} (u')) = u' \vee r_{X,t} \po r_{X,t} b$, which means that $\tau_{X}(\tau_{X\cup \{t\}} (u')) \po b$. But, as $a\cdot b$ is left-weighted, this implies that $a \so \tau_{X}(\tau_{X\cup \{t\}} (u'))$, which implies $ar_{t,X}\so u'$. This proves that $a r_{t,X} \cdot r_{X,t} b$ is left-weighted. 

\smallskip\noindent
For the other product, notice that $r_{X,t}a = \tau_{X\cup \{t\}}(\tau_{X} (a)) r_{X,t}$ and $b r_{t,X}= r_{t,X} \tau_{X\cup \{t\}}(\tau_{X} (b))$. As $r_{X,t}\cdot r_{t,X}$ is left-weighted (\autoref{ribbons_normal_form}) and the permutation induced by $   \tau_{X\cup \{t\} } \circ \tau_{X}$ preserves normal forms for elements in~$A_X$, $\tau_{X\cup \{t\}}(\tau_{X} (a)) \cdot \tau_{X\cup \{t\}}(\tau_{X} (b))$ is also left-weighted, and we can apply the same arguments as above.

\medskip
Therefore, we have proved that $\inf(yx)=\inf(y)=0$ and $\sup(yx)=\sup(y)=l$, as we wanted to show.
\end{proof}

%\begin{remark}
%Notice that, in the proof before, we can prove absorbability by using only the first claim: Since $\sup(y)=l$, $\sup(yx)$ is at least~$l$ and, as~$yx$ can be expressed in (\ref{eq3}) as the product of~$l$ simple elements, $\sup(yx)$ is at most~$l$. Hence $\sup(yx)=l$. However, we think that it is interesting to see that (\ref{eq3}) is exactly the normal form of~$yx$. 
%\end{remark}

\begin{remark}\label{ex1} Let us see why the later results do not work for $A_1$, $A_2$ and $I_{2m}$. Notice that, by definition, any absorbable element lies in $A^+$ or in $A^-$ (the negative monoid of $A$). We also recall that the only simple elements that are not absorbable are the ones of the form $\sigma_i^{-1}\Delta$ \citep[Example~1]{calvez2016graph}. 
For $A_1$ the set of absorbable elements is trivial and for $A_2$ the only absorbable elements are $\sigma_1$, $\sigma^{-1}_1$, $\sigma_2$ and $\sigma_2^{-1}$, so  $x$ can not be obtained as product of fewer than~$3\cdot |x|$ absorbable elements. Here $|\cdot|$ denotes the word length with respect to the set of standard Artin generators, $\Sigma$.

\medskip
On the other hand, the only absorbable elements in $I_{2m}$ are the absorbable simple elements and their inverses, that is, the elements of the form $$\underbrace{\sigma_1\sigma_2\sigma_1\cdots}_{p \text{ elements}} \quad \text{ or }\quad \underbrace{\sigma_2\sigma_1\sigma_2\cdots}_{p \text{ elements}}$$ with $p < m-1$ and their inverses. To see that no other element is absorbable, take an element in $I_{2m}$ of infimum 0 whose normal form is $s_1\cdot s_2$. By \citep[Lemma~3]{calvez2016graph} if $s_1\cdot s_2$ is absorbable, then it is absorbed by an element of infimum 0 with normal form $s_1'\cdot s'_2$. Let $s_1'\so\sigma_i \po s_2'$ and $s_1\so\sigma_j\po s_2$. If we suppose that $\sigma_1\po s_1$ (the case with $\sigma_2$ is analogous), then by absorbability $s_2'\so \sigma_2$. If $s_2's_1$ is not simple we would have that $\inf(s_1'\cdot s'_2\cdot s_1 \cdot s_2)>0$, contradicting absorbability. Hence, let us assume that $s_2's_1$ is a simple element of the form $\sigma_i\cdots \sigma_2\sigma_1\cdots \sigma_j$.  This implies that the normal form of $s_1'\cdot s'_2\cdot s_1 \cdot s_2$ is $s_1'\cdot s'_2 s_1 \cdot s_2$ of length 3, which also contradicts absorbability. Therefore, the length of the minimal expression of $x$ as a product of absorbable elements in $I_{2m}$ depends on~$|x|$.
\end{remark}

\medskip
The \emph{support} of a positive element~$u$, denoted~$supp(u)$, is the set of generators that appear in every positive word representing~$u$. Notice that $supp(\Delta_X)=X$, for any $X\subseteq \Sigma$. The next lemma is a generalization of the results \citep[Lemma~5.6]{Paris1997} and \cite[Lemma~2.2]{Godelle2003}, which use the classical concept of ribbon. 

\begin{lemma}\label{decomposition}
Let $(A,\Sigma)$ be an Artin--Tits system of spherical type, $X,Y\subsetneq \Sigma$ and $u,v \in A^+$ with $supp(u)=X$ and $supp(v)=Y$. Then any element $z\in A^+$ such that $z^{-1} u z = v$ can be written as $z=\alpha \beta $, where $\alpha \in A_X$ and~$\beta$ is a $X$--ribbon--$Y$.
\end{lemma}

\begin{proof}
We will use \citep[Proposition~6.3]{CGGW}, which says that if $c\po z$ is a minimal element conjugating~$u$ to a positive element (meaning that there is no $c'\po c$, $c'\neq 1,c$, such that $c'^{-1}uc'\in A^+$), then either $c\in A_X$ or $c= r_{t,X}$, for some $t\in \Sigma$ such that $t\notin X$. Using this, we can write $z$ as a product $c_1\cdots c_r$ where $c_1$ is a minimal conjugator from $v_0\coloneqq u$ to a positive element and 
$c_i$ is a minimal conjugator from $v_{i}\coloneqq (c_1\cdots c_{i-1})^{-1} u (c_1\cdots c_{i-1})$ to a positive element, for $1<i\leq r$. If we let $Y_i=supp(v_i)$, then either $c_i\in A_{Y_i}$ (type~1), or $c_i=r_{t,Y_i}$ for some $t\in \Sigma$ such that $t\notin Y_i$ (type~2).

\medskip
Suppose that we have some~$c_i$ of type 2 and~$c_{i+1}$ of type 1. In this case, $c_i c_{i+1}= c_{i+1}' c_i$, where $c_{i+1}'\in A_{Y_i}$, having an element of type~1 before an element of type~2. This allows us to arrange the product $c_1\cdots c_r$ to have $c_1\cdots c_r=\alpha \beta $, where $\alpha \in A_X$ and  $\beta$ is a $X$--ribbon--$Y$, as we wanted.
\end{proof}

\begin{proposition}\label{9abs}
Let $(A,\Sigma)$, $A\neq I_{2m}, A_1, A_2$, be an Artin--Tits system of spherical type, $X\subsetneq \Sigma$ and $u \in A^+$ with $supp(u)=X$. Then any element $x\in A$ such that $x^{-1} u x \in A^+$ is a product of at most 9 absorbable elements. 
\end{proposition}
\begin{proof}
Let $\Delta^k x_1\cdots x_r$ be the left normal form of~$x$.  Notice that $x_1\cdots x_r$ is a positive element that conjugates $\tau^k(u)\in A^+$ to a positive element. If we denote $X'=supp(\tau^k(u))$, by \autoref{decomposition} we have the decomposition $x_1\cdots x_r= \alpha \cdot \beta$, with $\alpha\in A_{X'}$ and $\beta$ a $X'$--ribbon--$Y$, where $Y:= supp(x^{-1}u x)$.

\medskip
Suppose that $\beta$ is a $X'$--ribbon--$Y$ of the form $r_1\cdots r_m$, satisfying $r_i\neq 1$ and $X_ir_i=r_iX_{i+1}$, where $X_i\subsetneq \Sigma$, for every $1\leq i \leq m$.  
Thanks to \citep[Theorem~5.1]{Paris1997} and the proof of \citep[Lemma~2.2]{TesisGodelle} we know that if $Z_1s=sZ_2$ for some $Z_1,Z_2\subsetneq \Sigma$ and $s\in A$, then $\Delta_{Z_1}s=s\Delta_{Z_2} $.
Hence, using the definition of left-ribbon we can write~$\beta$ in the form $\Delta_{X}^{-m}\Delta_{X_1\cup\{t_1\}}\cdots \Delta_{X_m\cup\{t_m\}}$, where $t_i\in \Sigma \setminus X_i$, for every $1\leq i \leq m$. Let $q\geq 0$ be the maximum number such that $\Delta^q\po\Delta_{X_1\cup\{t_1\}}\cdots \Delta_{X_m\cup\{t_m\}} $ and write $\beta'\coloneqq\Delta_{X_1\cup\{t_1\}}\cdots \Delta_{X_m\cup\{t_m\}}=\Delta^q \gamma,$ for some positive $\gamma$. By conjugating, we can ``move'' $\Delta^q$ to the left in order to have \[x=\Delta^{k+q}\cdot \tau^q(\alpha)\cdot \Delta_{Y'}^{-m}\gamma,\] where $\Delta_{Y'}=\tau^{q}(\Delta_{X'})$, so $Y'\subsetneq \Sigma$. As $\tau^q(\alpha)\in A_{Y'}$, by \autoref{deltak} and \autoref{diameter} the element $\Delta^{k+q}\cdot \tau^q(\alpha)$ is a product of at most 6 absorbable elements. Notice that $q\leq m$.

\medskip
We claim that $\Delta_{Y'}^{-m}\gamma$ is product of at most 3~absorbable elements. Observe that~$\Delta^q$ and~$\Delta_{X'}^m$ are prefixes of $\beta'$, hence $\Delta^q\vee \Delta_{X'}^m= \Delta_{X'}^{m-q}\Delta^q=\Delta^q\Delta_{Y'}^{m-q}$ is a prefix of $\beta'$. This means that $\Delta_{Y'}^{m-q}\po \gamma$ and then $\inf(\Delta_{Y'}^{q-m} \gamma)=0$. Then we decompose  $\Delta_{Y'}^{-m}\gamma=\Delta_{Y'}^{-q} \cdot \Delta_ {Y'}^{q-m}\gamma$. We know that $\Delta_{Y'}^{-q}$ is a product of at most 2~absorbable elements.
Finally, we claim that~$\Delta_{Y'}^{m-q}$ absorbs $\Delta_{Y'}^{q-m}\gamma$. Since, by construction, $\inf(\gamma)=0$, we just need to prove that $\sup(\gamma)= m-q$. Notice that $\Delta_{Y'}^{m-q}\po \gamma$ and $\Delta_{Y'}^{m-q}\not\po \Delta^p$,  for $1\leq p \leq m-q-1$,  hence $\sup(\gamma)$ is at least $m-q$. On the other hand, $\beta'\po \Delta^m$, so  $\gamma \po \Delta^{m-q}$ and $\inf(\gamma)\leq m-q$. Thus $\inf(\gamma)=m-q$, as we wanted to prove.
\end{proof}

\begin{corollary}\label{corol} Let $(A,\Sigma)$, $A\neq I_{2m}, A_1, A_2$, be an Artin--Tits system of spherical type and $X\subsetneq \Sigma$. The elements of $A$ normalizing $A_X$ are the product of at most $9$ absorbable elements.
\end{corollary}

\begin{proof}
In \citep[Lemma~7]{Cumplido2017b}, it is proven that $\alpha\in A$ normalizes~$A_X$ if~$\alpha$ commutes with an element, called central Garside element of~$A_X$, which turns to be a positive power of~$\Delta_X$. Therefore, \autoref{9abs} applies and~$\alpha$ is the product of at most $9$~absorbable elements.
\end{proof}

\noindent \emph{Proof of \autoref{th:main}}. We proceed as in the proof of \citep[Proposition~2]{calvez2016graph}. Firstly, notice that by definition some power of a periodic element~$x$ acts trivially on~$\ca(A)$. This means that~$x$ acts as finite-order isometries and then it acts elliptically on~$\ca(A)$. 

\medskip
If~$y$ is an element normalizing some standard parabolic subgroup~$A_X$, by \autoref{corol} the orbit of the trivial element by the action of~$y$ on $\ca(A)$ remains at distance at most~9 from the trivial element ($y^p$ normalizes $A_X$ for every $p>1$). Hence, $y$~acts elliptically on~$\ca(A)$. Finally, each element that normalizes the parabolic subgroup $P=\alpha^{-1}A_X\alpha$, for some $\alpha\in A$, is the conjugate of an element normalizing~$A_X$, and therefore it also acts elliptically on~$\ca(A)$. 

$\hfill \square$

\section[Free-product complement]{Free-product complement}
In this section we prove \autoref{thm: free product}, that is, the existence of a ``free-product complement'' for standard parabolic subgroups. The proof involves theory about groups acting on hyperbolic spaces. 

\subsection{WPD elements and elementary subgroups}
Let $G$ be a group acting on a hyperbolic metric space $\mathcal{H}$ by isometries. 

We say that $g\in G$ satisfies the  WPD (weak proper discontinuity) condition if for every $\epsilon > 0$ and every $v\in \mathcal{H}$, there exists $R=R(\epsilon)$ such that
$$\sharp\{h\in G : \dist_{\mathcal{H}}(v, hv)\leq \epsilon \text{ and }\dist_{\mathcal{H}}(g^R v, hg^R v)\leq \epsilon\}<\infty.$$

According to \citep[Lemma~6.5, Corollary~6.6]{DGO}, for every WPD element ${g\in G}$ acting loxodromically, there exists a unique maximal virtually cyclic subgroup, denoted by~$E_G(g)$, which consists of the elements that stabilize a quasi-geodesic axis for~$\langle g \rangle$. Moreover, it can be shown that
$$E_G(g)=\{h\in G \mid h{g}^i h^{-1}= {g}^j \text{ for some }i,j\in \mathbb{Z}\}.$$ Notice that the torsion-free elements in $E_G(g)$ are also loxodromic. Otherwise, the WPD condition would not be satisfied.

The starting point for proving \autoref{thm: free product} is the hyperbolicity of~$\ca(A)$ proved by Calvez and Wiest.  
We summarize the facts in \cite[Theorem~1]{calvez2016graph} and \cite[Theorems~1 \&~2, Propositions~5 \&~6, Remark~1]{Calvez2016} about $\ca(A)$ in the next theorem:
\begin{theorem}\label{thm: Cal is hyperbolic}
The additional length graph $\ca(A)$ associated to the classical Garside
structure of an irreducible Artin-Tits group of spherical type $A$ is $60$-hyperbolic.
Let $v=\langle \Delta \rangle$ be a vertex in $\ca(A)$. 
There is an element $g\in A/Z(A)$ satisfying
\begin{enumerate}
\item[(i)] $g$ has a preimage $\tilde{g}$ in $A^+$ of Garside length bounded by 12, that is $|\tilde{g}|_\mathcal{S}\leq 12$.
\item[(ii)] $g$ acts loxodromically and precisely $\dist_{\ca}(v, g^n  v)\geq n/2$.
\item[(iii)] $g$ is WPD and precisely for every $\kappa>0$ there is $N=N(\kappa)=4\kappa+319$ such that the cardinality of the set
$$\{h\in A/Z(A): \dist_{\ca}(v, hv)\leq\kappa, \dist_{\ca}(g^Nv, hg^Nv)\leq \kappa\}$$
is bounded above by $F(\kappa)=8\kappa +638$.
\end{enumerate}
\end{theorem}

\noindent
It is extremely important for our applications to notice that the constants involved in \autoref{thm: Cal is hyperbolic} are independent of the Artin-Tits group~$A$.

\subsection{A technical lemma}
The proof of \autoref{thm: free product} is a standard application of techniques of groups acting on hyperbolic spaces and it can be deduced easily from the results in \citep{DGO}.
However, it is not easy to trace back in the literature the exact dependency of the constants needed to have a unique constant~$K$ in the statement of \autoref{thm: free product}. %, and in many cases, one has to refer to details in the arguments in the proofs.
For that reason we will repeat some well-known arguments. 
The main point is showing that ``cancellations'' between products of elements of $\langle g^n, P\cdot Z(A)/Z(A)\rangle\subseteq A$ can be uniformly controlled when $n$ is large enough.

%Before proving \autoref{thm: free product} we collect some facts we will need.

\begin{lemma}\label{lem: tech WPD}
Let $A_X$ be a proper standard parabolic subgroup of an irreducible Artin-Tits group $A\neq A_1,A_2,I_{2m}$ of spherical type. Let $v=\langle \Delta \rangle$ be a vertex in~$\ca(A)$ and let~$g$ be the element in \autoref{thm: Cal is hyperbolic}. Denote by $\delta=60$ the hyperbolicity constant of $\ca (A)$ and $\dist:=\dist_{\ca}$. 
Let $a\geq 0$ and $n$ big enough (only depending on $a$, $g$ and $\ca(A)$). Then, for all $e, f\in \{g^{-1},g\}$, and for any non-trivial element $t\in (A_X\cdot Z(A))/Z(A)$ the following hold:

\begin{equation}\label{eq: n WPDa}
\dist(e^{n}v, tf^{ n}v)\geq\max\{\dist(v,e^{n}v),\dist(v,tf^{ n}v)\}+2\delta +a\text{.}
\end{equation}
\end{lemma}
\begin{proof}
We follow the argument of \citep[Proposition 6]{BestvinaFujiwara}. 
In fact, what we are going to show is that if \eqref{eq: n WPDa} does not hold for $n$ large enough, then  $t\in E_{A/Z(A)}(g)$. 
Since~$t$ stabilizes~$A_X$ and~$A_X$ is proper parabolic, \autoref{th:main} implies that~$t$ acts elliptically on~$\ca(A)$. 
On the other hand, $t$~is an infinite order element because $\langle \Delta \rangle \cap A_X=\{1\}$ and $t$ is a non-trivial element of  $(A_X\cdot Z(A))/Z(A)$. 
As $t$ lies in $E_{A/Z(A)}(g)$, we have that $t$ has to act loxodromically on~$\ca(A)$, which is a contradiction.

Notice that $\dist(v,e^nv)=\dist(v,f^n v)$ for any choice of $e,f\in \{g^{-1},g\}$.
Also for any element $t\in A_X\cdot Z(A))/Z(A)$, we have that $\dist(v,tv)\leq 9$ by \autoref{corol}.
By the triangle inequality, we have that  $\dist(v,tf^nv)\leq \dist (v,tv)+ \dist(tv, tf^nv)\leq 9 +\dist(v,f^nv)$. Thus, $|\dist(v, e^{n} v)-\dist(v,tf^{ n}v)|\leq 9$, and if \eqref{eq: n WPDa} does not hold, then $\dist(e^{n}v, tf^{ n}v)\leq\dist(v,e^{n}v)+2\delta +a+9$.

\medskip\noindent
Suppose that \eqref{eq: n WPDa} does not hold for some long enough~$n$ that will be specify later. 
Consider the geodesic $4$-gon in~$\ca(A)$ with geodesics $\gamma_0$ from~$v$ to $tv$, $\gamma_v$ from~$v$ to $e^{n}v$, $\gamma_{tv}$ from $tv$ to $tf^{ n}v$ and finally~$\gamma$ from $e^{n}v$ to $tf^{ n}v$ (see \autoref{4gon}).  
Note that~$\gamma_0$ has length at most~$9$ (\autoref{corol}), $\gamma$~has length less than  $\dist(v,e^{n}v)+2\delta+a+9$. 
%Recall that $9\leq 2\delta$. 
Let~$u_v$ be a vertex in~$\gamma_v$ that is the furthest one away from~$v$ with the property of being at distance at most~$2\delta$ of a vertex of~$\gamma_{tv}$.
Let~$u_{tv}$  be the vertex in~$\gamma_{tv}$ with $\dist(u_{v}, u_{tv})\leq 2\delta$.
%
%$$u_v\in \gamma_v \text{ s. t. } \dist(v, u_v)=\max\{\dist(v, w) \mid w \in \gamma_v, \dist(w, \gamma_{tv})\leq 2\delta\}$$
Note that $\dist(v,u_v)\leq 9+\dist(tv, u_{tv})+2\delta \leq \dist(v,u_v)+18+4\delta$.
Thus $\dist(v,u_v)$ and $\dist(tv,u_{tv})$ only differ by a constant independent of~$n$.

\begin{figure}[h]
  \centering
  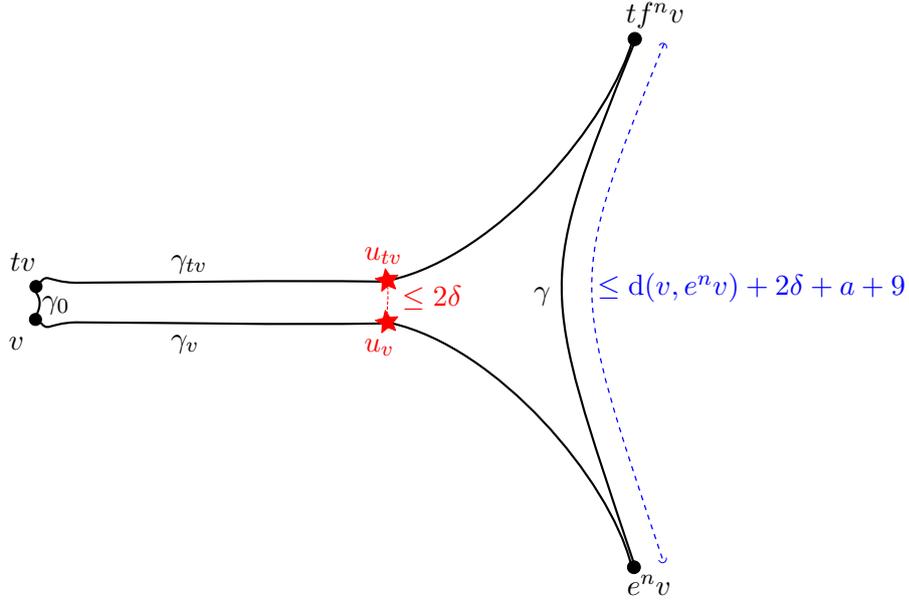
  \medskip
  \caption{4-gon in the proof of \autoref{thm: free product}.}
  \label{4gon}
\end{figure}

\medskip\noindent
We claim that $\dist(v,u_v)$ grows linearly with~$n$. 
 Indeed, since $u_v$ is the vertex in $\gamma_v$ furthest away from~$v$ with the property of being at distance at most~$2\delta$ of a vertex of~$\gamma_{tv}$, there must be a vertex~$u$ in~$\gamma$ such that $\dist(u_v,u)\leq 2\delta +1$.
If $\dist(u, e^{n}v)\leq \frac{1}{2}\dist(e^{n}v,tf^{ n}v)$, then 
$$\dist(u_v,e^{n}v)\leq 2\delta+1 +\frac{1}{2}\dist(e^{n}v,tf^{ n}v)\leq  2\delta +1+ \frac{1}{2}(\dist(v,e^nv)+2\delta+9+a).$$ 
Since $\dist(v,e^nv)=\dist(v,u_v)+\dist(u_v,e^nv)$, we get that $\dist(v,u_v)$ grows as $\frac{\dist(v,e^nv)}{2}\sim \frac{n}{4}.$ 
We similarly derive the same conclusion, if $\dist(u, tf^{ n}v)\leq \frac{1}{2}\dist(e^{n}v,tf^{ n}v)$, since  then $\dist(u_{tv},tf^{ n}v)\leq 4\delta+1 +\frac{1}{2}\dist(e^{n}v,tf^{ n}v)$, and $\dist(v,u_v)$ and $\dist(tv,u_{tv})$ only differ by a constant independent of~$n$.
Therefore, by increasing~$n$ we can make $\dist(v,u_v)$ as large as needed.

\medskip
Fix a geodesic paths~$p_0, p_0^{-1}$ in~$\ca(A)$ from~$v$ to~$gv$ and $v$ to~$g^{-1}v$, respectively.
Without loss of generality $p_0^{-1}$ is the path $g^{-1}p_0$ traversed backwards.
Since $|\tilde {g}|_\mathcal{S}\leq 12$, $p_0$ and $p_0^{-1}$ have length less than 12.
For $m> 0$, we write~$p(g^m)$ to denote the path consisting in concatenating $p_0, gp_0,\dots, g^{m-1} p_0$.
Similarly, for $m<0$, we write~$p(g^m)$ to denote the path consisting in concatenating $p^{-1}_0, g^{-1}p_0^{-1},\dots, g^{-m+1} p_0^{-1}$.
As~$g$ acts loxodromically, and $\dist(v, g^{\pm n}v)\geq n/2$,  the paths~$p(e^{n})$ and $p(f^{ n})$ are $(\lambda, c)$-quasi-geodesics where $\lambda= 24$ and $c=24$. 
Now, since $\ca(A)$ is $\delta$-hyperbolic, there is some constant $\varkappa=\varkappa(\delta,\lambda,c)=\varkappa(60,24,24)$ such that any $(\lambda,c)$-quasi-geodesic is in the $\varkappa$-neighbourhood of a geodesic path with same initial and final vertices.

\medskip\noindent
By increasing $n$, we can guarantee that $\gamma_v$ and $\gamma_tv$ have arbitrarily long initial subpaths that lie in the $2\delta$-neighbourhood of each other.  
Since~$p(e^{n})$ is  in the  $\varkappa$-neighbourhood of~$\gamma_v$ and~$tp(f^{ n})$ is  in the  $\varkappa$-neighbourhood of~$\gamma_{tv}$,  we can take arbitrarily long initial subpaths of~$p(e^{n})$ and~$tp(f^{ n})$, say~$p(e^{k})$ and~$tp(f^{ k})$, such that they lie in the $(2\varkappa+2\delta+\ell(p_0))$-neighbourhood of each other. Moreover, it is an standard argument to show that there must be a constant $D=D(\lambda,c, \varkappa,\delta, \ell(p_0))$ such that $p(e^{k})$ and $tp(f^{ k})$ synchronously $D$-fellow travel, and 
in particular $$\dist(e^{i} v, tf^{ i}v)\leq D \text{  for }i=0,1,\dots k.$$
\noindent Note that $D$ is independent of $k$, and $k$ can be made as large as needed. 

\medskip 
Let $N=N(2D)=4(2D)+319$ and $F=F(2D)=8(2D)+638$ as in \autoref{thm: Cal is hyperbolic}~(iii). 
Take $n$ large enough so that $k>F+N+1$.
By increasing $F$, if necessary, we assume $k=F+N+1$.
Let  $s=F+1$.
We are going to apply the WPD condition on $e^sv$ and $e^Ne^sv=e^kv$.
We are going to act by $tf^{ i}t^{-1}e^{-i}$ for $i=1,\dots, F+1$.
%Note that
%$$tf^{ i}t^{-1}e^{ -i} e^k v= tf^{ i}t^{-1} e^{k-i}v.$$
First,  since $\dist(tf^{ (k-i)}v, e^{(k-i)}v)\leq D$ and $\dist(tf^{ (s-i)}v, e^{(s-i)}v)\leq D$, 
%and $tf^{ i}t^{-1} tf^{ (k-i)}v=tf^{ k}v$, 
we have
$$\dist(tf^{ k}v, tf^{ i}t^{-1} e^{(k-i)}v) \leq D \quad \text{ and }\quad \dist(tf^{ s}v, tf^{ i}t^{-1} e^{(s-i)}v) \leq D.$$
As $\dist(tf^{ k}v, e^{k}v)\leq D$, we conclude that %e^{i}t^{-1} tp(e^{(k-i)})_+=tp(g^k)_+$, then
$$\dist((tf^{ i}t^{-1}e^{ -i})e^{k}v,e^{k}v )=\dist(tf^{ i}t^{-1}e^{k- i}v,e^{k}v)\leq \dist( tf^{ i}t^{-1} e^{(k-i)}v, tf^{ k}v)+ \dist(tf^{ k}v, e^{k}v) \leq 2D.$$
An analogous argument, using that $\dist(tf^{ s}v, e^{s}v)\leq D$, gives that $\dist((tf^{ i}t^{-1}e^{-i})e^sv, e^{s}v)\leq 2D$.
By the WPD condition, there are $1\leq i<j \leq F+1$ such that $tf^{ i}t^{-1} e^{- i}=tf^{ j}t^{-1}e^{- j}$ and thus, as $e,f\in \{g^{-1},g\}$, $tg^{i-j}t^{-1}=g^{i-j}$ or $tg^{i-j}t^{-1}=g^{j-i}$. Hence $t\in E_{A/Z(A)}(g)$, having the desired contradiction.
\end{proof}
\bigskip

\begin{corollary}\label{cor: WPD}
Let $A_X$ be a proper standard parabolic subgroup of an irreducible Artin-Tits group $A\neq A_1,A_2,I_{2m}$ of spherical type. 
Let $v=\langle \Delta \rangle$ be a vertex in~$\ca(A)$ and let~$g$ be the element in \autoref{thm: Cal is hyperbolic}. 
Denote by $\delta=60$ the hyperbolicity constant of $\ca (A)$ and $\dist:=\dist_{\ca}$. 
Let $a\geq 0$ and  $n$ big enough. Then, for all $\epsilon, \eta\in \{-1,1\}$, and for any non-trivial  $t,t'\in (A_X\cdot Z(A))/Z(A)$ the following hold:
\begin{equation}\label{eq: 2n-lox}
\dist(v,g^{2n}v)\geq \dist(v,g^nv)+2\delta+a\text{;}
\end{equation}
\begin{equation}\label{eq: n WPD}
\dist(g^{\epsilon n}v, tg^{\eta n}v)\geq\max\{\dist(v,g^{\epsilon n}v),\dist(v,tg^{\eta n}v)\}+2\delta +a\text{;}
\end{equation}
\begin{equation}\label{eq: rev1}
\dist(v,tg^{2n}v)\geq \max\{\dist(v,tg^n v),\dist(v,g^{n}v)\}+2\delta+a\text{;}
\end{equation}
\begin{equation}\label{eq: rev2}
\dist(g^{\epsilon n}t'v, tg^{\eta n}v)\geq\max\{\dist(v,g^{\epsilon n}t'v),\dist(v,tg^{\eta n}v)\}+2\delta +a.
\end{equation}
\end{corollary}
\begin{proof}
Claim \eqref{eq: 2n-lox}  follows from the fact that $g$ is loxodromic.
Claim \eqref{eq: n WPD} was proved in the previous lemma.

\medskip

Recall that by \autoref{corol}, $\dist(v,tv)\leq 9$ for all $t\in (A_X \cdot Z(A))/Z(A)$.
Let $a\in A/Z(A)$ and $u$ an arbitrary vertex. Observe the following:
 $\dist(atv, u)\leq \dist(atv,av)+\dist(av,u)=\dist(tv,v)+\dist(av,u)\leq 9 + \dist(av,u)$ and, similarly, $\dist(av,u)\leq \dist(av,atv)+\dist(atv, u)\leq 9 +\dist(atv, u)$. Therefore $$|\dist(atv, u)- \dist(av, u)|\leq 9.$$
Also observe that  $\dist(v,tav)=\dist(t^{-1}v,a v)\leq \dist(t^{-1}v,v)+\dist(v,av)\leq 9 + \dist(v,av)$ so $\dist(v,tav)-\dist(v,av)\leq 9$.  Similarly,  $\dist(v,av)\leq \dist(v,t^{-1}v)+\dist(t^{-1}v,av)\leq 9+\dist(v,tav)$ so $\dist(v,av)-\dist(v,tav)\leq 9$. Therefore  $$|\dist(v,tav)-\dist(v,av)|\leq 9.$$

For claim  \eqref{eq: rev1},  we have that $|\dist(v,g^{2n}v)-\dist(v,tg^{2n}v)|\leq 9$ and 
$\dist(v,g^{n}v)+2\delta+a+9 \geq \max\{\dist(v,tg^{n} v),\dist(v,g^{n}v)\}.$
Thus \eqref{eq: rev1}  follows from the fact that $g$ is loxodromic.
Finally, claim \eqref{eq: rev2} follows from  the previous lemma and   $$|\dist(g^{\epsilon n}t'v, tg^{\eta n}v)- \dist(g^{\epsilon n}v, tg^{\eta n}v)|\leq 9,$$
and  
$$|\max\{\dist(v,g^{\epsilon n}t'v),\dist(v,tg^{\eta n}v)\}-\max\{\dist(v,g^{\epsilon n}v),\dist(v,tg^{\eta n}v)\}|\leq 9.$$
\end{proof}

\subsection{Proof of \autoref{thm: free product}}
The strategy to prove \autoref{thm: free product} relies on the application of the following lemma. 

\begin{lemma}[{\citealp[Lemma~1.1]{Delzant}}]\label{lem:Delzant}
Let $(x_i)$ be a sequence of points on a $\delta$-hyperbolic geodesic metric space such that
$\dist(x_{i+2},x_i)\geq \max \left\{\dist(x_{i+2},x_{i+1}), \dist(x_{i+1},x_{i})\right\}+2\delta+a$. 
Then  $\dist(x_i,x_j)\geq a |j-i|$.
\end{lemma}

\noindent
We proceed now to prove that there is a free product complement:

\begin{proof}[Proof of \autoref{thm: free product}]
Let $v=\langle \Delta \rangle$ be a vertex in~$\ca(A)$ and let~$g$ be the element in \autoref{thm: Cal is hyperbolic}. 
Fix $a>0$ and  $n$ be big enough so that  the conclusions of \autoref{cor: WPD} hold.

\smallskip

We are going to prove that a reduced word~$w$ over $\mathcal{A}=\langle g^{n}, g^{-n}\rangle \cup   (A_X\cdot Z(A))/Z(A)$ represents a non trivial element of $A/Z(A)$.
Indeed, if $w$ is a reduced word over  $\mathcal{A}$ then it has the form 
$$w=t_1g^{k_1n}t_2 g^{k_2n} \dots t_s g^{k_s n},$$ 
with $k_i\in \mathbb{Z}\setminus\{0\}$ for $i<s$ and $k_s\in \mathbb{Z}$,  and $t_i\in (A_X\cdot Z(A))/Z(A)\setminus \{1\}$ for $i>1$, and $t_1\in (A_X\cdot Z(A))/Z(A)$. 
If $s=1$, we already know that $w$ can not represent the identity since $\langle g \rangle$ and $(A_X\cdot Z(A))/Z(A)$ have trivial intersection. 
Thus we can assume that $s>1$.
If $t_1=1$, conjugating $w$ by $g_1^{k_1 n}$ and then performing free reductions, we obtain a shorter word  $w'=t_1'g^{k_1'n}t_2' g^{k'_2n} \dots t'_{s'} g^{k'_{s'} n}$ (respect to the alphabet $\mathcal{A}$) with the new $t'_1\neq 1$. 
Moreover $w$ represents the trivial element if and only if $w'$ does. 
So we will assume that $t_1\neq 1$. 
With a similar argument, we can assume that $k_s\neq 0$.

\medskip
Let $\mathrm{sgn}\colon \mathbb{Z} \rightarrow \{-1,0,1\}$ be the sign function. 
We can consider the orbit of~$v$ under prefixes of~$w$. 
Since $t_1\neq 1$ and $k_s\neq 0$ we can view $w$ as a word over  $\{g^n, g^{-n}, tg^{n}, tg^{-n} : t\in  A_X\cdot Z(A)/Z(A)\}$, having  the following  sequence of vertices $\{x_i\}_{i=0}^{\sum |k_i|}$ in $\ca$
\[
\begin{array}{l}
x_0=v,\, x_1=(t_1g^{\mathrm{sgn}(k_1)n}) v,\, x_2=(t_1g^{\mathrm{sgn}(k_1)2n})v, \, \dots \,, x_{|k_1|}=(t_1g^{k_1n})v,
\\ \\
x_{|k_1|+1}=(t_1 g^{k_1n}) (t_2 g^{\mathrm{sgn}(k_2)n})v ,\, x_{|k_1|+2}=(t_1 g^{k_1n}) (t_2 g^{\mathrm{sgn}(k_2)2n})v,\, \dots\, , x_{|k_1|+|k_2|}= (t_1 g^{k_1n}) (t_2 g^{k_2n})v,\\ \\
\dots \\ \\
x_{\sum |k_i|} = t_1g^{k_1n}t_2 g^{k_2n} \dots t_s g^{k_s n} v = wv.

\end{array}
\bigskip
\]
We want to check that this sequence of vertices satisfies the hypothesis of \autoref{lem:Delzant}. We can suppose that $x_i=hv$. Then we have several possibilities:

\begin{itemize}

\item Assume that $x_{i+1}= hg^{\epsilon n}v$, with $\epsilon\in\{-1,1\}$. If $x_{i+2}=hg^{ 2\epsilon n}v$, use \eqref{eq: 2n-lox} to check the hypothesis inequality. If on the other hand $x_{i+2}=hg^{\epsilon n}t_jv g^{\eta  n}$, with $\eta \in \{-1,1\}$, use \eqref{eq: n WPD}. 

\item Assume that $x_{i+1}= ht_jg^{\epsilon n}v$. If $x_{i+2}=ht_jg^{ 2\epsilon n}v$, use \eqref{eq: rev1} to check the desired inequality. 
If on the contrary $x_{i+2}=ht_jg^{\epsilon  n}t_{j+1}v g^{\eta  n}$, with $\eta \in \{-1,1\}$, use \eqref{eq: rev2}. 

\end{itemize}

\noindent
Then, by \autoref{lem:Delzant} we have that 
 $$\dist(v, wv)=\dist(x_0, x_{\sum |k_i|})\geq a \left(\sum_{i=1}^{s} |k_i|\right).$$
In particular $\dist(v,wv)=0$ if and only if $w=t_1$ and hence $w$, since $t_1\neq 1$, $w$ does not represent the trivial element

\medskip

Hence, we have that  $(\langle A_X, g^n \rangle\cdot Z(A))/Z(A)\leqslant A/Z(A)$ is isomorphic to $((A_X \cdot Z(A))/Z(A)) *((\langle g^n\rangle \cdot Z(A))/Z(A))$.
Since~$Z(A)$ is generated by a power of $\Delta$ and we have that $\langle \Delta \rangle \cap A_X=\{1\}=\langle g^n\rangle \cap \langle \Delta \rangle$, we get that $(\langle A_X, g^n\rangle \cdot Z(A))/Z(A)\leqslant A/Z(A)$ is isomorphic to $A_X*\langle g^n \rangle$. 
Fix a pre-image~$\tilde{g}$ of~$g$ in~$A$. 
We know that~$\tilde{g}$ can be taken to be positive with length at most~12.
Finally, notice that $A_X*\langle\tilde{g}^n\rangle$ maps onto $\langle A_X,\tilde{g}^n\rangle \leqslant A$, which maps onto $(\langle A_X,g^n\rangle\cdot Z(A))/Z(A)\cong A_X *\langle g^n\rangle$. 
Therefore, $\langle A_X, \tilde{g}^n\rangle\cong A_X*\langle \tilde g^n \rangle$ and $g_*=\tilde g^n$. 
\end{proof}

\begin{remark}
By \autoref{corol}, $N_A(A_X)$ (the normalizer of~$A_X$ in~$A$) acts elliptically with a diameter bounded by~$9$. 
Note that $\Delta^2\in N_A(A_X)$ and therefore
$\langle N_Z(A), \tilde{g}^n \rangle$ cannot be a free product since~$\Delta^2$ commutes with~$\tilde{g}$ and 
hence $\langle N_Z(A), \tilde{g}^n \rangle$ has non-trivial center.
However, if one takes $H\leqslant N_A(A_X)$ such that $H\cap Z(A)=\{1\}$ then
our proof shows that $\langle H, \tilde{g}^n\rangle \cong H* \langle \tilde{g}^n\rangle$.
\end{remark} 

\section{Exponential growth rate}
In this last section we prove the \autoref{cor: growth} about exponential growth rate of parabolic subgroups with respect to  the Garside generating set.

\medskip
The coproduct in the category of monoids is constructed in the same way as in groups.
If $A,B$ are two monoids, their coproduct is denoted by~$A*B$. Its elements are reduced words in $A\cup B$ and its operation is the concatenation (followed with reduction).
If~$M$ is a monoid and~$T$ is some subset, we write~$\langle T \rangle^+$ for the submonoid genererated by~$T$\footnote{The notation $\langle \cdot \rangle^+$ usually is reserved for the sub-semigroup generated by. Note that the subsemigroup generated by a set $T$ and the submonoid generated by~$T$ just differ by one element and thus the growth rates are equal.}.
To prove \autoref{cor: growth} we need the following lemma, which is a standard application of generating functions (see \citealp[VI.A.Proposition 4]{delaHarpe}), modified to give some weight to a free generator. 
\begin{lemma}\label{lem: weighted free product}
Let~$G$ be a group, $\mathcal{S}$ be a finite generating set, and $M$ be a sub-monoid. Let $1\leq\alpha\leq\omega(M,\mathcal{S})$.
Suppose that $g\in G$, satisfies that $|g|_\mathcal{S}\leq k$ and $\langle M, g\rangle^+ \cong M *\langle g \rangle^+$.
Let also~$\gamma$ be the positive root of $1-\alpha x- x^k$.
Then $\frac{1}{\gamma}>\alpha$ is a lower bound for  $\omega(\langle M, g\rangle^+, \mathcal{S})$.

\smallskip\noindent
In particular, if~$M$ is a subgroup, since $\langle M,g\rangle$ contains $\langle M, g\rangle^+$ we have that  $\frac{1}{\gamma}>\alpha$ is a lower bound for 
 $\omega(\langle M, g\rangle, \mathcal{S})$.
\end{lemma}

\begin{proof}
Let us denote by $\beta_M(n)=\sharp\{h\in M : |h|_\mathcal{S}\leq n\}$. 
Since $\beta_M$ is  sub-multiplicative, Fekete's lemma implies that 
$$\alpha\leq\omega(M,\mathcal{S})=\lim_{n\to \infty} \sqrt[n]{\beta_M(n)}=\inf_{n\geq 1}\sqrt[n]{\beta_M(n)}$$
exists and therefore $\alpha^n\leq \beta_M(n)$ for all $n\geq 1$.

\medskip\noindent
Now, an element of  $t\in \langle M, g\rangle^+$  has a unique expression written in the form
\begin{equation}\label{eq:generic element}
m_1 g^{n_1} m_2 g^{n_2} m_3\dots m_{\ell} g^{n_{\ell}}
\end{equation}
where $m_i\in M$, $m_i\neq 1$ for $i>1$, $n_i\in \mathbb{Z}_{\geq 0}$, $n_i\neq 0$, for $i\leq \ell$. The $\mathcal{S}$-length of the element in \eqref{eq:generic element} is bounded above by
$$||t||\coloneqq \sum_{i=1}^\ell |m_i|_\mathcal{S}+ |g|_\mathcal{S}\cdot \left(\sum_{i=1}^\ell n_i \right).$$
Notice that $|t|_\mathcal{S}\leq ||t||$.
Thus \begin{equation}\label{eq: bounding with norm}
\beta_{\langle M, g\rangle^+}(n)\geq \beta^{||\cdot ||}_{\langle M,g\rangle^+}(n)\coloneqq\sharp\{ t \in \langle M, g\rangle^+ : ||t||\leq n\}.
\end{equation}
We are going to estimate the growth rate of  $\beta^{||\cdot ||}_{\langle M,g\rangle^+}(n)$.
For that, for $s\in \mathbb{N}$ let $L(s)\subseteq \langle M, g \rangle^+$ denote the subset of elements of $\langle M, g\rangle^+$ that when written in the form \eqref{eq:generic element} one has that $\ell=s$, that is
$$L(s)=\{m_1 g^{n_1} \dots m_{s} g^{n_{s}}\in \langle M,g\rangle^+ :  (m_i\in M, n_i\in \mathbb Z_{\geq 0}, \colon \forall i), (m_i\neq 1: i>1), (n_i\neq 0: i<s)\}.$$ 
Since $\bigsqcup_{s\geq 1}L(s)=\langle M, g\rangle^+$, letting $\beta^{|| \cdot ||}_{L(s)}(n)$ denote $\sharp \{t\in L(s): || t || \leq n\}$ we have that
\begin{equation}\label{eq: partition}
\beta^{||\cdot ||}_{\langle M,g\rangle^+}(n)= \sum_{s\geq 0} \beta^{|| \cdot ||}_{L(s)}(n).
\end{equation}
For a function $\beta\colon \mathbb{N}\to \mathbb{N},$ we denote by
$\mathcal{G}_\beta(x)$ the growth series $\sum_{n\geq 0}\beta(n) x^n$.
Let also
$$\mathcal{G}_1(x)\coloneqq\sum_{n\geq 0} \alpha^n x^n=\frac{1}{1-\alpha x}\quad\text{ and }\quad\mathcal{G}_2(x)\coloneqq\sum_{n\geq 1}(x^{|g|_\mathcal{S} })^n=\frac{1}{1-x^{|g|_\mathcal{S}}}.$$
Since $\beta_M(n)\geq \alpha^n$, we have that $\mathcal{G}_{\beta_M}(x)\geq \mathcal{G}_1(x)$. 
Also observe $\beta_{\langle g\rangle^+}(n)\geq 1$ and thus $\mathcal{G}_{\beta_{\langle g \rangle^+}}(x)\geq \mathcal{G}_2(x)$.
Then \begin{align*}
\mathcal{G}_{\beta^{|| \cdot ||}_{L(s)}}(x)&=  \mathcal{G}_{\beta_M}(x)((\mathcal{G}_{\beta_{\langle g \rangle^+}}(x)-1)(\mathcal{G}_{\beta_M}(x)-1))^{s-1} \mathcal{G}_{\beta_{\langle g \rangle^+}}(x)\\
&\geq \mathcal{G}_1(x)((\mathcal{G}_2(x)-1)(\mathcal{G}_1(x)-1))^{s-1} \mathcal{G}_2(x),
\end{align*}
for $s\geq 1$. 
Using \eqref{eq: bounding with norm} and \eqref{eq: partition} and the previous inequality we get that
%Thus $\mathcal{G}_{\beta^{||\cdot ||}_{\langle M,g\rangle^+}}$ is coefficient-wise smaller than
\begin{align*}
\mathcal{G}_{\beta^{||\cdot ||}_{\langle M,g\rangle^+}} & \geq \sum_{s\geq 1} \mathcal{G}_{\beta^{|| \cdot ||}_{L(s)}}(x)\\
& \geq \sum_{s\geq 1}\mathcal{G}_1(x)((\mathcal{G}_2(x)-1)(\mathcal{G}_1(x)-1))^{s-1} \mathcal{G}_2(x) \\ & =  \mathcal{G}_1(x) \mathcal{G}_2(x)\sum_{s\geq 1}((\mathcal{G}_2(x)-1)(\mathcal{G}_1(x)-1))^{s-1} \\
& = \mathcal{G}_1(x) \mathcal{G}_2(x) \dfrac{1}{1-(\mathcal{G}_2(x)-1)(\mathcal{G}_1(x)-1)}\\
& =\dfrac{1}{1-\alpha x -x^{|g|_\mathcal{S}}.}
\end{align*}

\noindent
Therefore, the exponential growth rate of $\beta^{||\cdot ||}_{\langle M,g\rangle^+}(n)$ is bounded below by the inverse of the convergence radius of the growth series of 
$(1-\alpha x -x^{|g|_\mathcal{S}})^{-1}$ which is the positive root~$\gamma_0$ of $1-\alpha x -x^{|g|_\mathcal{S}}$.
 It is easy to see that $1/\gamma_0$  is strictly greater than~$\alpha$ and moreover if $|g|_\mathcal{S}\leq k$, then the positive root $\gamma$  of $1-\alpha x-x^k$ is bigger than~$\gamma_0$ (an thus $1/\gamma<1/\gamma_0$ is a lower bound for the growth rate).
\end{proof}

\noindent
To finally prove \autoref{cor: growth}, we need one more result. This is a result that is well-known for Artin--Tits groups of spherical type (see an explanation in \citealp[Section~3]{CGGW}), which has been generalized for Garside groups in \citep[Theorem~1.13]{Godelle2007}:
\begin{lemma}\label{thm: Godelle}
Let $A$  be an Artin--Tits group of spherical type, $A_X$ be a parabolic subgroup and~$\mathcal{S}$ the set of Garside generators of~$A$.  Then $\mathcal{S}\cap A_X$ is the set of Garside generators of~$A_X$ and the embedding of~$A_X$ into~$A$ with respect to the Garside generators is isometric, that is  for every  $g\in A_X$ the length of $g$ with respect to $\mathcal{S}^{\pm }$ and $\mathcal{A}\cap \mathcal{S}^{\pm 1}$ is the same.
%More precisely, if $g$ is written in left greedy normal form $b^{-1}a$ with $a,b$ words over $\mathcal{S}^*$ then, $a,b$ only use letters from $A_X\cap \mathcal{S}^+$.
\end{lemma}

\begin{proof}[Proof of \autoref{cor: growth}]
Let~$A$ be an Artin--Tits group of spherical type and~$\mathcal{S}$ be its set of Garside generators of~$A$. 
Let $A_X$ be a proper parabolic subgroup.
Suppose first that $A=A_1$. 
Then~$A$ is cyclic and the only proper parabolic subgroup is trivial.  
It is easy to see that $\omega(A_X,\mathcal{S}^{\pm 1})=\omega(A,\mathcal{S}^{\pm 1})=1$.
Now suppose that $A$~is either either equal to~$A_2$ or~$I_{2m}$.
Both of these groups (and their monoids) have exponential growth since they contain non-abelian free semigroups. 
Therefore $\omega(A,\mathcal{S}^{\pm 1})>1$.
However, the only proper parabolic subgroups are trivial or cyclic and thus $\omega(A_X,\mathcal{S}^{\pm 1})< \omega(A,\mathcal{S}^{\pm 1})$.

\medskip\noindent
Henceforth, we assume that $A\neq A_1,A_2, I_{2m}$. By \autoref{thm: free product}, there is a constant~$K$ and an element $g\in A$ satisfying $|g|_\mathcal{S}\leq K$ such that
$\langle A_X,g\rangle \cong A_X *\langle g \rangle$. By \autoref{lem: weighted free product} we have that $\omega(A_X,\mathcal{S}^{\pm 1})<\omega(\langle A_X, g \rangle ,\mathcal{S}^{\pm 1}).$ By definition, we also have $\omega(\langle A_X, g \rangle ,\mathcal{S}^{\pm 1})\leq \omega(A,\mathcal{S}^{\pm 1})$. 
Notice that, since $\langle A_X,g\rangle \cong A_X *\langle g \rangle$, the submonoid $\langle A_X^+,g\rangle^+$ is isomorphic to the coproduct of monoids $A_X^+ *\langle g \rangle^+$ and similarly $\omega(A_X^+, \mathcal S) <\omega(A^+,\mathcal{S})$.
This proves the first claim of the corollary. 

\medskip
\noindent
Let us show that the sequence $\{\omega(A_n,\mathcal{S}^{\pm 1}_{A_n})\}_{n=1}$ goes to infinite. The proofs for the other 5 claims are analogous. 
Consider each $A_i$ sitting inside $A_{i+1}$ as a standard parabolic subgroup $A_1\leqslant A_2\leqslant A_3 \leqslant\dots$ .
Observe that by \autoref{thm: Cal is hyperbolic} there are $g_i\in A_i$, for $i=1,2,3,\dots$, satisfying $|g_i|_{\mathcal{S}_{A_i}}\leq K$ and such that $\langle A_{i-1}, g_{i}, g_{i+1},\dots, g_{i+s}\rangle \leqslant A_{i+s}$ is isomorphic to $A_{i-1}*\langle g_i \rangle * \langle g_{i+1}\rangle * \dots *\langle g_{i+s}\rangle$.
Thus we can use \autoref{lem: weighted free product} inductively to get a lower bound for $\omega(A_{i+s},\mathcal{S}_{A_{i+s}}^{\pm 1})$. 
Notice that, by \autoref{thm: Godelle}, we have $\omega(A_i,\mathcal{S}^{\pm 1}_{A_i})=\omega(A_{i}, \mathcal{S}^{\pm 1}_{A_{n}})$ for every~$n\geq i$.
Let $\alpha_1=\omega(A_1,\mathcal{S}_{A_1}^{\pm 1})=1$. 
For $i\geq 1$, let $\gamma_i$ be the root of $1-\alpha_i x-x^K$
and let $\alpha_{i+1}=1/\gamma_{i}$.
By the previous discussion and \autoref{lem: weighted free product}, $\alpha_i\leq \omega(A_i, \mathcal{S}_{A_i}^{\pm 1})$.
So it is enough to show that $\{\alpha_i\}$ goes to infinity.
Note that also by \autoref{lem: weighted free product}, $\alpha_i<\alpha_{i+1}$ for all $i\in \mathbb{N}$.
So the sequence~$\{\alpha_i\}$ is increasing and either converges or goes to infinity. 
If it converges to some value, say~$\eta$, then $1/\eta$ must be a root of $0=1-{\eta}x- x^K$. But this means that $0=\frac{1}{\eta}^K$. 
Therefore, $\{\alpha_i\}_{i=1}^\infty$ diverges and then $\{\omega(A_n,\mathcal{S}^{\pm 1}_{A_n})\}_{n=1}$ goes to infinity.
\end{proof}

\bigskip

\noindent{\textbf{\Large{Acknowledgments}}}

\medskip

The first author acknowledges partial support from the Spanish Government through grants number MTM2014-53810-C2-01, MTM2017-82690-P and through the ``Severo Ochoa Programme for Centres of Excellence in R\&{}D” (SEV–2015–0554). The second author acknowledges the support of the projects MTM2016-76453-C2-1-P (financed by Spanish Government and FEDER) and FQM-2018 (financed by Junta de Andaluc\'ia). She also thanks the University of Burgundy for financing her with a one-year postdoc contract. We thank Bert Wiest for suggesting to prove \autoref{th:main} and Jos\'e M. Conde-Alonso for useful discussions. We also thank Juan Gonz\'alez-Meneses for reading the first part of this article and helping to improve it. 

We are very grateful to the referee of this paper for their careful reading and thoughtful report. 

\medskip
\bibliography{Parabolic_on_CAL}

\bigskip\bigskip{\footnotesize%

\textit{ Yago Antol\'{i}n, Departamento de  Matem\'aticas,
Universidad Aut\'onoma de Madrid and 
Instituto de Ciencias Matem\'aticas, CSIC-UAM-UC3M-UCM, 
28049 (Madrid), Spain} \par
 \textit{E-mail address:} \texttt{\href{mailto:yago.anpi@gmail.com}{yago.anpi@gmail.com}} 

 \medskip

\textit{ Mar\'{i}a Cumplido, IMB, UMR 5584, CNRS, Univ. Bourgogne Franche-Comt\'e, 21000 Dijon, France} \par
 \textit{E-mail address:} \texttt{\href{mailto:maria.cumplido.cabello@gmail.com}{maria.cumplido.cabello@gmail.com}}
 }

\end{document}